\documentclass[12pt]{article}

\usepackage{amsthm,amssymb,amsmath,amsfonts}
\usepackage{amscd}
\usepackage{latexsym}

\usepackage{graphicx}
\usepackage{epstopdf}
\usepackage{float}
\usepackage{tikz}
\usetikzlibrary{arrows.meta,calc}

\usepackage{xcolor}

\usepackage{titlesec}
\usepackage{titletoc}

\usepackage[numbers,square]{natbib}

\usepackage[font=normalsize,labelsep=none]{caption}

\usepackage{enumitem}
\usepackage{enumerate}

\usepackage{setspace}
\usepackage{changepage}

\usepackage{mdframed}
\usepackage{framed}
\usepackage{fancybox}

\usepackage{diagbox}
\usepackage{arydshln}

\usepackage{bbm}

\usepackage{algorithm}
\usepackage{algorithmic}

\usepackage{empheq}

\usepackage{verbatim}

\usepackage[
colorlinks=true,
linkcolor=blue,
citecolor=blue,
urlcolor=blue,
pagebackref=true
]{hyperref}

\newcommand{\EXCL}{\mathsf{\scriptstyle EXCL}}
\newcommand{\NEXCL}{\mathsf{\scriptstyle NEXCL}}
\usepackage[
colorlinks=true,
linkcolor=blue,
citecolor=blue,
urlcolor=blue,
pagebackref=true
]{hyperref}

\makeatletter
\renewcommand*{\backref}[1]{}
\renewcommand*{\backrefalt}[4]{%
	\ifcase #1\relax
	\else
	\unskip~{\textup{#2}}%
	\fi
}

\makeatother

 \newtheorem{theorem}{Theorem}[section]
 
 \newtheorem{lemma}{Lemma}[section]

\usepackage{xcolor}
\usepackage{tikz}
\usetikzlibrary{arrows.meta,calc}

\usepackage{enumitem}
\setenumerate[1]{itemsep=0pt,partopsep=0pt,parsep=\parskip,topsep=2pt}
\setitemize[1]{itemsep=0pt,partopsep=0pt,parsep=\parskip,topsep=2pt}
\setdescription{itemsep=0pt,partopsep=0pt,parsep=\parskip,topsep=2pt}

\usepackage{mdframed}
\mdfsetup{
	linewidth=0.75pt,        
}

\newtheoremstyle{myremark}
{3pt}    
{3pt}    
{\normalfont} 
{}       
{\normalfont\bfseries} 
{.}      
{0.5em}  
{}       
\theoremstyle{myremark}
\newtheorem{remark}{Remark}[section]

\titleformat{\section} 
{\raggedright\large\bfseries\bf} 
{\thesection .\quad} 
{0pt} 
{}  
 
\titleformat{\subsection} 
{\raggedright\normalsize\bfseries\it} 
{\thesubsection .\quad} 
{0pt} 
{} 

\numberwithin{equation}{section}

\newlength{\boxedparwidth}
  {\begin{center} \begin{tabular}{|@{\hspace{.315in}}c@{\hspace{.15in}}|}
                  \hline \\ \begin{minipage}[t]{\boxedparwidth}
                  \setlength{\parindent}{.25in}}%
  {\end{minipage} \\ \\ \hline \end{tabular} \end{center}}

\begin{document}
\captionsetup[figure]{labelfont={bf},labelformat={default},labelsep=period,name={Fig.},font={footnotesize}}
\captionsetup[table]{labelfont={bf},labelformat={default},labelsep=period,name={Table},font={footnotesize}}

\title{The sorting-Denert statistic}
\author{Shao-Hua Liu\\
	School of Statistics and Data Science\\
	Guangdong University of Finance and Economics\\
	Guangzhou, China\\
	\texttt{liushaohua@gdufe.edu.cn}}
\date{}
\maketitle

 \vskip 0mm

\noindent {\bf Abstract.}
Denert's statistic is a classical Mahonian statistic on permutations.
Together with the excedance number,
it forms an Euler--Mahonian pair,
a result first conjectured by Denert and later proved by Foata and Zeilberger.
Motivated by this classical result,
we introduce the sorting-Denert statistic,
a sorting-index analogue of Denert's statistic.
It is obtained by replacing,
in Denert's statistic,
the inversion number of the non-excedance-letter subsequence
with its sorting index.
We prove that the sorting-Denert statistic,
together with the excedance number,
forms a new Euler--Mahonian pair.

   \vskip 2mm
\noindent {\bf Keywords}: 
Sorting-Denert statistic,
Denert's statistic,
sorting index,
Euler--Mahonian statistic,
excedance number

   \vskip 2mm
\noindent {\bf 2020 Mathematics Subject Classification}: 05A05, 05A19.
   \vskip 2mm

\titlecontents{section}[1.5em]
{\footnotesize \vspace{-10pt}}
{\contentslabel{1.5em}}{\hspace*{-1.6em}}
{~\titlerule*[0.6pc]{$.$}~\contentspage}

\titlecontents{subsection}[3.5em]
{\footnotesize \vspace{-10pt}}
{\contentslabel{2.3em}}{\hspace*{-4em}}
{~\titlerule*[0.6pc]{$.$}~\contentspage}

\section{Introduction}
Let $\mathfrak{S}_{n}$ denote the set of permutations of $[n]:=\{1,2,\ldots,n\}$.
Given a permutation $\sigma=\sigma_{1}\sigma_{2}\ldots\sigma_{n}\in\mathfrak{S}_{n}$,
a position $i$, $1\leqslant i\leqslant n-1$, is called a \emph{descent} of $\sigma$ if $\sigma_{i}>\sigma_{i+1}$.
Let $\textsf{Des}(\sigma)$ be the set  of descents of $\sigma$,
and let $\textsf{des}(\sigma)=|\textsf{Des}(\sigma)|$,
where $|\cdot|$ denotes cardinality.
A position $i$, $1\leqslant i\leqslant n$, is called 
an \emph{excedance} of $\sigma$,
or an \emph{excedance position}, 
if $\sigma_{i}>i$.
Let $\textsf{Exc}(\sigma)$ be the set of excedances of $\sigma$,
and let $\textsf{exc}(\sigma)=|\textsf{Exc}(\sigma)|$.
It is well known that $\textsf{des}$ and $\textsf{exc}$ are equidistributed over $\mathfrak{S}_{n}$,
and a permutation statistic equidistributed with them is said to be \emph{Eulerian}.

Let $\sigma=\sigma_{1}\sigma_{2}\ldots\sigma_{n}\in\mathfrak{S}_{n}$.
A pair $(i,j)$ of positions  is called an \emph{inversion} of $\sigma$ if $i<j$ and $\sigma_{i}>\sigma_{j}$. 
Let $\textsf{inv}(\sigma)$ be the number of inversions of $\sigma$.
Define the \emph{major index} of $\sigma$, denoted by $\textsf{maj}(\sigma)$,  by
\[\textsf{maj}(\sigma)=\sum_{i\in\textsf{Des}(\sigma)}i.\]
MacMahon \cite{MacMahon-1916} showed that \textsf{inv} and \textsf{maj} are equidistributed over $\mathfrak{S}_{n}$, and that
\begin{align*} 
	\sum_{\sigma\in\mathfrak{S}_{n}}q^{\textsf{inv}(\sigma)}=\sum_{\sigma\in\mathfrak{S}_{n}}q^{\textsf{maj}(\sigma)}=[n]_{q}!,
\end{align*}
where  
$[n]_{q}!=[n]_{q}[n-1]_{q}\ldots[1]_{q}$
with 
$[n]_{q}=1+q+\cdots+q^{n-1}.$
In his honor, any permutation statistic with this distribution over $\mathfrak{S}_{n}$ is said to be \emph{Mahonian}.
We recall two Mahonian statistics:
the sorting index and Denert's statistic.

The sorting index was introduced by Petersen \cite{Petersen-2011} for permutations and was later naturally extended to words by Grady and Poznanovi\'{c} \cite{Grady-2018}.
The sorting index can be described as the total distance that the elements of $\sigma$ travel when $\sigma$ is sorted using the Straight Selection Sort
algorithm \cite{Knuth-1998}. 
More precisely, let
$\sigma=\sigma_1\sigma_2\ldots\sigma_m$ be a permutation of a finite set
$A=\{a_1<a_2<\cdots<a_m\}$ of positive integers. 
We sort $\sigma$ into $a_1a_2\ldots a_m$ as follows: 
using a transposition, we first move the largest element $a_m$ to its proper position, 
then the second largest element $a_{m-1}$ to its proper position, and so on. Whenever $a_r$ is moved from position $i$ to its proper position $r$, this step contributes $r-i$. 
The \emph{sorting index} of $\sigma$, denoted by  $\textsf{sor}(\sigma)$, is the sum of these contributions.
For example, the steps for sorting $\sigma=142683$ are
\[
1\,4\,2\,6\,\underline{8}\,\underline{3}
\xrightarrow{(56)}
1\,4\,2\,\underline{6}\,\underline{3}\,8
\xrightarrow{(45)}
1\,\underline{4}\,2\,\underline{3}\,6\,8
\xrightarrow{(24)}
1\,\underline{3}\,\underline{2}\,4\,6\,8
\xrightarrow{(23)}
1\,2\,3\,4\,6\,8
\]
and thus
$\textsf{sor}(142683)=(6-5)+(5-4)+(4-2)+(3-2)=5.$

We now recall Denert's statistic,
a classical Mahonian statistic introduced by Denert \cite{Denert-1990}.
Here we use an equivalent definition due to Foata and Zeilberger \cite{Foata-1990}.
Let $\sigma=\sigma_{1}\sigma_{2}\ldots \sigma_{n}\in\mathfrak{S}_{n}$.
If $i$ is an excedance,
we call $\sigma_{i}$ an \emph{excedance letter}.
Let $\EXCL(\sigma)$
be the subsequence of $\sigma$ consisting of the excedance letters,
and let $\NEXCL(\sigma)$ be the subsequence of $\sigma$ consisting of the remaining letters (i.e., the non-excedance letters).
Define \emph{Denert's statistic} of $\sigma$, denoted by $\textsf{den}(\sigma)$, by
\[~\textsf{den}(\sigma)=\sum_{i\in\textsf{Exc}(\sigma)}i+\textsf{inv}(\EXCL(\sigma))+
\textsf{inv}(\NEXCL(\sigma)).\]

Motivated by the above definition of Denert's statistic, 
we define a new Mahonian statistic by replacing the inversion number of the non-excedance-letter subsequence with its sorting index. 
More precisely, define the \emph{sorting-Denert statistic}
of $\sigma$, denoted by $\textsf{sden}(\sigma)$, by
\[\textsf{sden}(\sigma)=\sum_{i\in\textsf{Exc}(\sigma)}i+\textsf{inv}(\EXCL(\sigma))+
\textsf{sor}(\NEXCL(\sigma)).\]

For example, if \(\sigma=715492683\), then
\[
\textsf{Exc}(\sigma)=\{1,3,5\},\qquad
\EXCL(\sigma)=759,\qquad
\NEXCL(\sigma)=142683,
\]
and hence
\[
~\textsf{den}(\sigma)
=
1+3+5+\textsf{inv}(759)+\textsf{inv}(142683)
=
14,
\]
while
\[
\textsf{sden}(\sigma)
=
1+3+5+\textsf{inv}(759)+\textsf{sor}(142683)
=
15.
\]

A pair of permutation statistics is said to be \emph{Euler--Mahonian} if it is equidistributed with $(\textsf{des},\textsf{maj})$.
Denert \cite{Denert-1990} conjectured that the pair $(\textsf{exc},\textsf{den})$ is Euler--Mahonian.
The first proof of Denert's conjecture was given by Foata and Zeilberger \cite{Foata-1990},
and a direct bijective proof was later provided by Han \cite{Han-1990-direct}.
Recently, Liu \cite{Liu-2025} and Huang--Wen--Yan \cite{Yan-2026} gave two new bijective proofs.
The Huang--Wen--Yan bijection applies not only to permutations but also to words.

The following theorem, which gives a new Euler--Mahonian pair, is the main result of this paper.
\begin{theorem}\label{Thm-exc-sden-is-Euler-Mahonian}
The pair $(\emph{\textsf{exc}},\emph{\textsf{sden}})$ is Euler--Mahonian.
\end{theorem}
\begin{remark}
	In the definition of $\textsf{sden}$,
	the replacement is made only on the non-excedance letters.
	This choice is essential in the sense that the corresponding replacement
	on the excedance letters does not preserve the Mahonian distribution,
	regardless of whether the replacement on the non-excedance letters is also made.
\end{remark}
The rest of this paper is organized as follows.
In Section \ref{section-Construction of the insertion map}, 
we construct the insertion map \(\phi_n^{\mathrm{sden}}\).
In Section \ref{section-proof-of-bijection-exc-sden}, 
we prove that this map is a bijection and satisfies the required property,
thereby proving the main theorem.

\section{Construction of the insertion map}\label{section-Construction of the insertion map}
Motivated by the Huang--Wen--Yan bijection \cite{Yan-2026},  
we present a bijective proof of Theorem \ref{Thm-exc-sden-is-Euler-Mahonian}.
To prove that a pair of permutation statistics \((\textsf{st}_{1},\textsf{st}_{2})\) is Euler--Mahonian,
it suffices to show that there exists a bijection
\vspace{-6pt}
 \[\phi_{n}:\mathfrak{S}_{n-1}\times\{0,1,\ldots,n-1\}\rightarrow\mathfrak{S}_{n}\]
 
\vspace{-10pt}
\noindent such that for $\sigma\in\mathfrak{S}_{n-1}$ and $0\leqslant c\leqslant n-1$, 
we have
\begin{equation}\label{eq-Euler--Mahonian}
	\begin{split}
		\textsf{st}_{1}(\phi_{n}(\sigma,c))&=\left
		\{
		\begin{aligned}
			&\textsf{st}_{1}(\sigma), \quad\quad\quad\text{if~}0\leqslant c\leqslant \textsf{st}_{1}(\sigma),\\
			&\textsf{st}_{1}(\sigma)+1,\quad\text{~otherwise,}
		\end{aligned}
		\right.\\				
		\textsf{st}_{2}(\phi_{n}(\sigma,c))&=\textsf{st}_{2}(\sigma)+c, 
	\end{split}
\end{equation}
see \cite{Rawlings-1981,Han-1990-direct}. 
The main result of this section is the following theorem.
\begin{theorem}\label{Thm-bijection-exc-sden}
	There exists a bijection 
	\[
	\phi_{n}^{\emph{\text{sden}}}: \mathfrak{S}_{n-1}\times\{0,1,\ldots,n-1\}\rightarrow\mathfrak{S}_{n},
	\]
	such that for $\sigma\in\mathfrak{S}_{n-1}$ and $0\leqslant c\leqslant n-1$, 
	we have 
\begin{equation}\label{eq-(exc,sden)-Exc-sden}
	\begin{split}
	\emph{\textsf{Exc}}(\phi_{n}^{\emph{\text{sden}}}(\sigma,c))&=
	\left\{
	\begin{aligned}
		&\emph{\textsf{Exc}}(\sigma), \quad\quad\quad\quad~\text{if~}0\leqslant c\leqslant \emph{\textsf{exc}}(\sigma),\\
		&\emph{\textsf{Exc}}(\sigma)\cup\{k_{d}\},\quad\text{~otherwise}
	\end{aligned}
	\right.\\
	\emph{\textsf{sden}}(\phi_{n}^{\emph{\text{sden}}}(\sigma,c))&=\emph{\textsf{sden}}(\sigma)+c,
		\end{split}
\end{equation}
where, in the second case of the first formula, $d=c-\emph{\textsf{exc}}(\sigma)$ and \(k_d\) is the \(d\)-th smallest non-excedance position of \(\sigma\).
\end{theorem}
The first formula in \eqref{eq-(exc,sden)-Exc-sden} is a strengthening of the first formula in \eqref{eq-Euler--Mahonian}.
Thus, the bijection $\phi_{n}^{\text{sden}}$ in Theorem \ref{Thm-bijection-exc-sden} yields a bijective proof of Theorem \ref{Thm-exc-sden-is-Euler-Mahonian}.
Before introducing the bijection, 
we need the following auxiliary notation.

Let \(\tau=\tau_{1}\tau_{2}\ldots\tau_{m}\) be a permutation of a finite set
\(A=\{a_1<a_2<\cdots<a_m\}\) of positive integers.
We view \(\tau\) as the bijection on \(A\) defined by \(\tau(a_i)=\tau_i\)
for \(1\leqslant i\leqslant m\).
For an integer \(e\) and an element \(a\in A\) with \(a\leqslant e\),
define
\[
f_{\tau}(a,e)=\tau^{j}(a),
\]
where \(j\) is the smallest positive integer satisfying
\[
\tau^{j}(a)\leqslant e.
\]
This is well-defined.
Indeed,
some positive iterate of \(a\) under the permutation \(\tau\) returns to \(a\),
which is at most \(e\).

For example,
when \(\tau=1\,\,2\,\,5\,\,11\,\,6\,\,12\,\,4\),
we write it in two-line notation as
\[
\tau=
\left(
\begin{matrix}
	1 & 2 & 4 & 5 & 6 & 11 & 12 \\
	1 & 2 & 5 & 11 & 6 & 12 & 4
\end{matrix}
\right).
\]
Let $e=7$. Then
\[
4 \xrightarrow{\tau} 5 \leqslant 7,
\qquad
5 \xrightarrow{\tau} 11 \xrightarrow{\tau} 12 \xrightarrow{\tau} 4 \leqslant 7,
\qquad
6 \xrightarrow{\tau} 6 \leqslant 7.
\]
Thus,
\begin{align}\label{example-b1b2b3}
	f_{\tau}(4,7)=5,
	\qquad
	f_{\tau}(5,7)=4,
	\qquad
	f_{\tau}(6,7)=6.
\end{align}

We next introduce the \textsf{sden}-labeling of the spaces of a permutation
$\sigma\in\mathfrak{S}_{n-1}$.
Let $\textsf{exc}(\sigma)=s$. 
We label the spaces of $\sigma$ as follows. First, label the space after the last letter of $\sigma$ with $0$. 
Next, label the spaces immediately preceding the excedance letters with
\[
1,2,\ldots,s
\]
from right to left. 
Finally, label the spaces immediately preceding the non-excedance letters with
\[
s+1,s+2,\ldots,n-1
\]
from left to right.

For example, the \textsf{sden}-labeling of 
$\sigma=3\,\,10\,\,1\,\,14\,\,7\,\,2\,\, 8\,\,9\,\,5\,\,13\,\,11\,\,6\,\,12\,\,4$ is as follows:
\begin{align}\label{example-labeling}
\sigma={}_{7}\textcolor{red}{3}\,
{}_{6}\textcolor{red}{10}\,
{}_{8}\textcolor{blue}{1}\,
{}_{5}\textcolor{red}{14}\,
{}_{4}\textcolor{red}{7}\,
{}_{9}\textcolor{blue}{2}\,
{}_{3}\textcolor{red}{8}\,
{}_{2}\textcolor{red}{9}\,
{}_{10}\textcolor{blue}{5}\,
{}_{1}\textcolor{red}{13}\,
{}_{11}\textcolor{blue}{11}\,
{}_{12}\textcolor{blue}{6}\,
{}_{13}\textcolor{blue}{12}\,
{}_{14}\textcolor{blue}{4}_{0}.
\end{align}
In this example, we use red letters to indicate excedance letters and blue letters to indicate non-excedance letters.
The spaces immediately preceding the red letters are labeled with $1,\ldots,7$ from right to left, 
while the spaces immediately preceding the blue letters are labeled with $8,\ldots,14$ from left to right. 
Finally, the space after the last letter is labeled with $0$.

We are now ready to define the map
\[
\phi_{n}^{\mathrm{sden}}:
\mathfrak{S}_{n-1}\times\{0,1,\ldots,n-1\}\rightarrow\mathfrak{S}_{n}.
\]
Let $\sigma=\sigma_{1}\sigma_{2}\ldots\sigma_{n-1}\in\mathfrak{S}_{n-1}$ and let $c$ be an integer with $0\leqslant c\leqslant n-1$.
We distinguish three cases by considering the space labeled with $c$.

\medskip
\noindent\textbf{Case 1: The space labeled with $c$ is the last space.}
Set
\[
\phi_{n}^{\mathrm{sden}}(\sigma,c)
=
\sigma_1\sigma_2\ldots\sigma_{n-1}n.
\]

\medskip
\noindent\textbf{Case 2: The space labeled with $c$ is immediately before an excedance letter.}
We obtain
\(\phi_{n}^{\mathrm{sden}}(\sigma,c)\) through the following steps.

\smallskip
\noindent\textit{Step 1: Shift the excedance letters, extract \(e\), and insert \(n\).}

Starting from the excedance letter immediately after the space labeled with \(c\),
shift the excedance letters one excedance position to the right, 
stopping either at the first excedance letter that would become a non-excedance letter after the shift, 
or at the last excedance letter if no such excedance letter exists. 
Denote this stopping excedance letter by \(e\),
and extract it from the sequence.
Fill the gap vacated by the first shifted letter with \(n\).
Let $\sigma^{\prime}$ be the resulting sequence.

For example, consider the permutation $\sigma$ in (\ref{example-labeling}).
For $c=6$, we have  
\begin{center}
	\includegraphics[width=0.70\textwidth]{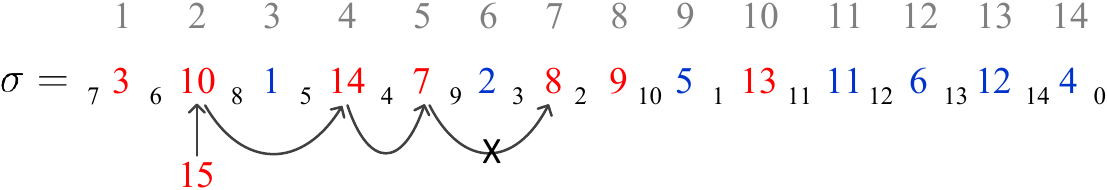}
\end{center}
In this example, the gray numbers indicate the positions of the corresponding letters.
Thus,  $e=7$, and the resulting sequence is 
\begin{align*}
	\sigma^{\prime}=
	\textcolor{red}{3}\,\,\,
	\textcolor{red}{15}\,\,\,
	\textcolor{blue}{1}\,\,\,
	\textcolor{red}{10}\,\,\,
	\textcolor{red}{14}\,\,\,
	\textcolor{blue}{2}\,\,\,
	\textcolor{red}{8}\,\,\,
	\textcolor{red}{9}\,\,\,
	\textcolor{blue}{5}\,\,\,
	\textcolor{red}{13}\,\,\,
	\textcolor{blue}{11}\,\,\,
	\textcolor{blue}{6}\,\,\,
	\textcolor{blue}{12}\,\,\,
	\textcolor{blue}{4}.
\end{align*}

\smallskip
\noindent\textit{Step 2. Shift the non-excedance letters and insert \(e\).}

Extend \(\sigma'\) by appending an empty position at the right end.
Then shift the non-excedance letters
lying weakly to the right of position \(e\) in \(\sigma'\),
or equivalently,
weakly to the right of the letter \(\sigma'_e\),
one non-excedance position to the right.
Fill the gap vacated by the first shifted letter with \(e\).
Let \(\sigma^{\prime\prime}\) be the resulting permutation.

In our running example, we have
\begin{center}
	\includegraphics[width=0.74\textwidth]{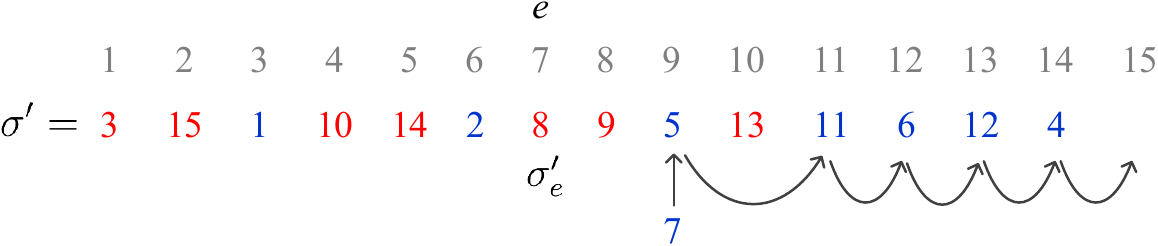}
\end{center}
The resulting permutation is 
\begin{align}\label{example-sigma-prime-prime}
	\sigma^{\prime\prime}=
	\textcolor{red}{3}\,\,\,
	\textcolor{red}{15}\,\,\,
	\textcolor{blue}{1}\,\,\,
	\textcolor{red}{10}\,\,\,
	\textcolor{red}{14}\,\,\,
	\textcolor{blue}{2}\,\,\,
	\textcolor{red}{8}\,\,\,
	\textcolor{red}{9}\,\,\,
	\textcolor{blue}{7}\,\,\,
	\textcolor{red}{13}\,\,\,
	\textcolor{blue}{5}\,\,\,	
	\textcolor{blue}{11}\,\,\,
	\textcolor{blue}{6}\,\,\,
	\textcolor{blue}{12}\,\,\,
	\textcolor{blue}{4}.
\end{align}

\smallskip
\noindent\textit{Step 3.  Adjust non-excedance letters.}

Consider the original permutation \(\sigma\).
Let \(x\) be the number of non-excedance letters of \(\sigma\)
that lie weakly to the right of position \(e\) in \(\sigma\),
or equivalently,
weakly to the right of the letter \(\sigma_e\),
and are less than \(e\).
As will be noted in the proof of Lemma \ref{lemma:Exc-sden},
we have \(x\geqslant1\).
Let \(\tau\) be the subsequence of \(\sigma\) consisting of the non-excedance letters.
That is,
\[
\tau=\NEXCL(\sigma).
\]
List the letters of \(\tau\) that are smaller than \(e\) in increasing order,
and let \(b_1<b_2<\cdots<b_x\) be the last \(x\) entries in this list.
For \(1\leqslant i\leqslant x\),
set
\[
f_i=f_\tau(b_i,e).
\]
Obtain \(\phi_n^{\text{sden}}(\sigma,c)\) from \(\sigma^{\prime\prime}\)
by cyclically replacing the letters \(f_1,f_2,\ldots,f_x\) as follows:
\[
f_1\mapsto f_2\mapsto \cdots \mapsto f_x\mapsto f_1.
\]

In our example,
we have
\[
\tau=
\textcolor{blue}{1}\,\,
\textcolor{blue}{2}\,\,
\textcolor{blue}{5}\,\,
\textcolor{blue}{11}\,\,
\textcolor{blue}{6}\,\,
\textcolor{blue}{12}\,\,
\textcolor{blue}{4},
\qquad
e=7.
\]
In \(\sigma\), the non-excedance letters lying weakly to the right of position \(7\) are
\(
5,11,6,12,4.
\)
Among them,
the letters smaller than \(7\) are
\(
5,6,4.
\)
Thus \(x=3\).
The letters of \(\tau\) that are smaller than \(7\),
listed in increasing order,
are
\(
1,2,4,5,6.
\)
Hence the last \(x=3\) entries in this list are
\[
b_1=4,
\qquad
b_2=5,
\qquad
b_3=6.
\]
By \eqref{example-b1b2b3},
we have
\[
f_1=5,
\qquad
f_2=4,
\qquad
f_3=6.
\]
The permutation \(\phi_{15}^{\text{sden}}(\sigma,6)\) is obtained from
\(\sigma^{\prime\prime}\) in \eqref{example-sigma-prime-prime}
by performing the cyclic replacement
\[
5\mapsto 4\mapsto 6\mapsto 5.
\]
Hence
\[
\phi_{15}^{\text{sden}}(\sigma,6)
=
\textcolor{red}{3}\,\,\,
\textcolor{red}{15}\,\,\,
\textcolor{blue}{1}\,\,\,
\textcolor{red}{10}\,\,\,
\textcolor{red}{14}\,\,\,
\textcolor{blue}{2}\,\,\,
\textcolor{red}{8}\,\,\,
\textcolor{red}{9}\,\,\,
\textcolor{blue}{7}\,\,\,
\textcolor{red}{13}\,\,\,
\textcolor{blue}{4}\,\,\,
\textcolor{blue}{11}\,\,\,
\textcolor{blue}{5}\,\,\,
\textcolor{blue}{12}\,\,\,
\textcolor{blue}{6}.
\]

\medskip
\noindent\textbf{Case 3: The space labeled with $c$ is immediately before a non-excedance letter.}
Extend \(\sigma\) by appending an empty position at the right end.
Starting from the non-excedance letter immediately after the space labeled with \(c\),
shift the non-excedance letters one non-excedance position to the right.
Fill the gap vacated by the first shifted letter with $n$.
Let $\phi_{n}^{\text{sden}}(\sigma,c)$ be the resulting permutation.

For example, consider the permutation $\sigma$ in (\ref{example-labeling}). 
For \(c=9\), we have
\begin{center}
	\includegraphics[width=0.71\textwidth]{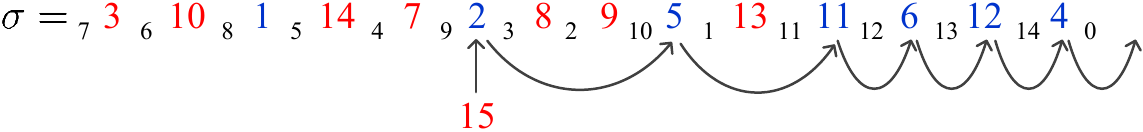}
\end{center}
The resulting permutation is
\begin{align*}
	\phi_{15}^{\text{sden}}(\sigma,9)=
	\textcolor{red}{3}\,\,\,
	\textcolor{red}{10}\,\,\,
	\textcolor{blue}{1}\,\,\,
	\textcolor{red}{14}\,\,\,
	\textcolor{red}{7}\,\,\,
	\textcolor{red}{15}\,\,\,
	\textcolor{red}{8}\,\,\,
	\textcolor{red}{9}\,\,\,
	\textcolor{blue}{2}\,\,\,
	\textcolor{red}{13}\,\,\,
	\textcolor{blue}{5}\,\,\,
	\textcolor{blue}{11}\,\,\,
	\textcolor{blue}{6}\,\,\,
	\textcolor{blue}{12}\,\,\,
	\textcolor{blue}{4}.
\end{align*}

\section{Proof of Theorem \ref{Thm-bijection-exc-sden}}\label{section-proof-of-bijection-exc-sden}
In this section,
we prove that the map \(\phi_n^{\text{sden}}\) constructed in the previous section
is a bijection and satisfies \eqref{eq-(exc,sden)-Exc-sden}.
This proves Theorem \ref{Thm-bijection-exc-sden},
and hence yields Theorem \ref{Thm-exc-sden-is-Euler-Mahonian}.
Throughout this section,
we use the notation introduced in the previous section.
We first prove \eqref{eq-(exc,sden)-Exc-sden},
and then show that \(\phi_n^{\text{sden}}\) is a bijection.
Before doing so,
we need several lemmas.

Let \(\tau\) be a permutation of a finite set of positive integers,
and let \(e\) be an integer.
Define \(\tau_{(e)}\) as follows:
perform the Straight Selection Sort algorithm only on the letters greater than \(e\),
moving them to their proper positions,
and then delete these letters from the resulting sequence.

For example, let
\(
\tau=1\,\,2\,\,5\,\,11\,\,6\,\,12\,\,4
\)
and let \(e=7\).
The letters greater than \(7\) are \(11\) and \(12\).
Thus
\[
\tau=
1\,\,2\,\,5\,\,11\,\,6\,\,\underline{12}\,\,\underline{4}
\xrightarrow{(67)}
1\,\,2\,\,5\,\,\underline{11}\,\,6\,\,\underline{4}\,\,12
\xrightarrow{(46)}
1\,\,2\,\,5\,\,4\,\,6\,\,11\,\,12
\xrightarrow{\text{delete }11,12}
1\,\,2\,\,5\,\,4\,\,6
=
\tau_{(7)}.
\]
\begin{lemma}\label{lemma-eta_e-induced-sequence}
Let \(\tau\) be a permutation of a finite set
\(A=\{a_1<a_2<\cdots<a_m\}\) of positive integers,
and let \(e\) be a positive integer.
	Suppose that
	\[
	a_1<a_2<\cdots<a_k
	\]
	are all the letters of $A$ that are less than or equal to \(e\). Then
	\[
	\tau_{(e)}
	=
	f_{\tau}(a_1,e)f_{\tau}(a_2,e)\ldots f_{\tau}(a_k,e).
	\]
\end{lemma}

\begin{proof}
	We view \(\tau\) as a bijection on \(A\).
	If \(k=0\), then both sides are empty. If \(k=m\), then no letter is greater than \(e\), so \(\tau_{(e)}=\tau\), and \(f_{\tau}(a_i,e)=\tau(a_i)\) for all \(i\). Thus the result is immediate in these two boundary cases.
	
	We first describe the effect of one step of the Straight Selection Sort algorithm
	on the cycle structure.
	During Straight Selection Sort,
	consider the step in which the letter \(h\) is treated.
	At this moment,
	all letters greater than \(h\) have already been placed in their proper positions.
	Write the current state in two-line notation as
	\[
	\left(
	\begin{matrix}
		a_1 & a_2 & \cdots & a_m\\
		\eta_1 & \eta_2 & \cdots & \eta_m
	\end{matrix}
	\right).
	\]	
	If the entry below \(h\) is \(h\),
	then this step changes nothing.
	In terms of cycles,
	the fixed point
	\[
	h\mapsto h
	\]
	remains unchanged.
	Otherwise,
	suppose that the letter \(h\) appears below \(u\),
	and let \(v\) be the entry below \(h\).
	Since all letters greater than \(h\) have already been placed in their proper positions,
	we must have \(u<h\).
	Thus this step has the following local form:
	\[
	\left(
	\begin{matrix}
		\ast & u & \ast & h & \ast \\
		\ast & h & \ast & v & \ast
	\end{matrix}
	\right)
	\longrightarrow
	\left(
	\begin{matrix}
		\ast & u & \ast & h & \ast \\
		\ast & v & \ast & h & \ast
	\end{matrix}
	\right).
	\]
    Equivalently, in the cycle structure,
    the segment
    \[
    u\mapsto h\mapsto v
    \]
    is replaced by
    \[
    u\mapsto v,
    \qquad
    h\mapsto h.
    \]
    That is,
    \(h\) is removed from its original cycle
    and becomes a fixed point.
	Thus, in both cases,
	after the step in which \(h\) is treated,
	the letter \(h\) becomes, or remains, a fixed point.
	
    Applying this observation successively to all letters greater than \(e\),
    we see that each such letter is removed from its original cycle and becomes a fixed point.
    After these fixed points are suppressed,
    the induced bijection on the letters less than or equal to \(e\) is obtained from the cycle decomposition of \(\tau\) by deleting all letters greater than \(e\).
	
	Hence, for each \(a_i\) with \(1\leqslant i\leqslant k\),
	the image of \(a_i\) in this induced bijection is the first positive iterate of \(a_i\) under \(\tau\) that is less than or equal to \(e\).
	By definition, this image is \(f_{\tau}(a_i,e)\).
	Thus,
	viewed as a permutation of \(\{a_1,a_2,\ldots,a_k\}\),
	\(\tau_{(e)}\) has the following two-line notation:
	\[
	\tau_{(e)}
	=
	\left(
	\begin{matrix}
		a_1 & a_2 & \cdots & a_k\\
		f_{\tau}(a_1,e) & f_{\tau}(a_2,e) & \cdots & f_{\tau}(a_k,e)
	\end{matrix}
	\right).
	\]
	Therefore,
	in one-line notation,
	\[
	\tau_{(e)}
	=
	f_{\tau}(a_1,e)f_{\tau}(a_2,e)\ldots f_{\tau}(a_k,e).
	\]
	This proves the lemma.
\end{proof}
We now focus on Case 2 and recall the relevant notation:
\begin{itemize}[leftmargin=2.8em]
	\item \(e\) is the excedance letter extracted in Step 1.
	
    \item \(x\) is the number of non-excedance letters of \(\sigma\) lying weakly to the right of position \(e\) in \(\sigma\) and less than \(e\).
	
	\item 
	\(
	\tau=\NEXCL(\sigma).
	\)
	
    \item \(b_1<b_2<\cdots<b_x\) are the largest \(x\) letters of \(\tau\) that are less than \(e\), and \(f_i=f_{\tau}(b_i,e)\) for \(1\leqslant i\leqslant x\).
\end{itemize}

\begin{lemma}\label{lemma-sor-increase-case2}
	In Case 2,
	let \(\widetilde{\tau}\) be the subsequence of
	\(\phi_{n}^{\emph{\text{sden}}}(\sigma,c)\) consisting of the non-excedance letters.
	Then
	\[
	\emph{\textsf{sor}}(\widetilde{\tau})
	=
	\emph{\textsf{sor}}(\tau)+x.
	\]
\end{lemma}
\begin{proof} 
	In \(\sigma\), let
	\[
	\{g_1,g_2,\ldots,g_x\}
	\quad\text{and}\quad
	\{h_1,h_2,\ldots,h_y\}
	\]
	be the sets of non-excedance letters lying weakly to the right of position
	\(e\), with
	\[
	g_i<e,
	\qquad
	h_j>e.
	\]
	Thus \(g_1,\ldots,g_x\) are exactly the non-excedance letters counted by
	\(x\).  
	We first record two observations.
	
	First, the last \(x+y\) letters of \(\tau\) are precisely \(g_1,\ldots,g_x,h_1,\ldots,h_y\) as a set.
	We call this suffix the \emph{final block}.
	
	Second, the letters of \(\tau\) that are greater than \(e\) are precisely \(h_1,\ldots,h_y\). Indeed, if \(h>e\) is a letter of \(\tau\), 
	then \(h\) is a non-excedance letter of \(\sigma\).
	Thus its position in \(\sigma\) is at least \(h\),
	and hence is greater than \(e\).
	Therefore \(h\) lies to the right of position \(e\) in \(\sigma\),
	so \(h\in\{h_1,\ldots,h_y\}\).

	It follows that \(h_1,\ldots,h_y\) are the largest \(y\) letters of \(\tau\),
	and hence their proper positions in \(\tau\) are the last \(y\) positions.
	Since these letters lie in the final block and their proper positions are also contained in this block,
	the sorting steps involving them affect only this block.
	After these letters have been sorted and then deleted,
	the last \(x\) letters of \(\tau_{(e)}\) are precisely \(g_1,\ldots,g_x\), in some order.
	
	We now determine this order.
	Since \(e\) is an excedance letter of \(\sigma\),
	whereas \(\tau\) consists of the non-excedance letters of \(\sigma\),
	we have \(e\notin\tau\).
	Let
	\[
	a_1<a_2<\cdots<a_k
	\]
	be all the letters of \(\tau\) that are less than \(e\).
	Then these are also all the letters of \(\tau\) that are less than or equal
	to \(e\).
	By Lemma \ref{lemma-eta_e-induced-sequence},
	\[
	\tau_{(e)}
	=
	f_{\tau}(a_1,e)f_{\tau}(a_2,e)\ldots f_{\tau}(a_k,e).
	\]
	Since \(b_1<b_2<\cdots<b_x\) are the largest \(x\) letters among
	\(a_1,\ldots,a_k\),
	we have
	\[
	b_1=a_{k-x+1},
	\quad
	b_2=a_{k-x+2},
	\quad
	\ldots,
	\quad
	b_x=a_k.
	\]
	Hence \(\tau_{(e)}\) ends with
	\[
	f_{\tau}(b_1,e)f_{\tau}(b_2,e)\ldots f_{\tau}(b_x,e)
	=
	f_1f_2\ldots f_x.
	\]
	Write
	\[
	\tau_{(e)}=P\,f_1f_2\ldots f_x
	\]
	for some possibly empty sequence \(P\).
	
    We now record how Steps 2 and 3 affect the non-excedance letters.
    After Step 2, each shifted non-excedance letter occupies a position strictly to the right of its original position, and hence remains a non-excedance letter. 
    Moreover, the inserted letter \(e\) occupies the position vacated by the first shifted non-excedance letter. 
    This position lies weakly to the right of position \(e\). Hence the inserted letter \(e\) is also a non-excedance letter. 
    It remains to check that Step 3 does not change this property. Since the
    last \(x\) letters of \(\tau_{(e)}\) are precisely \(g_1,\ldots,g_x\) as
    a set, we have
    \[
    \{f_1,\ldots,f_x\}=\{g_1,\ldots,g_x\}.
    \]
    The cyclic replacement
    \[
    f_1\mapsto f_2\mapsto \cdots \mapsto f_x\mapsto f_1
    \]
    therefore only permutes the non-excedance letters \(g_1,\ldots,g_x\). These letters all occupy positions weakly to the right of position \(e\) in the full permutation, and each of them is less than \(e\). 
    After the cyclic replacement, each affected letter still occupies a position weakly to the right of position \(e\). Hence each affected letter is less than its new position, and therefore remains a non-excedance letter.

    Consequently, \(\widetilde{\tau}\) is obtained from \(\tau\) by first inserting \(e\) immediately before the final block, namely the last \(x+y\) letters of \(\tau\), and then applying the cyclic replacement to the letters \(f_1,f_2,\ldots,f_x\).

	We compare the Straight Selection Sort procedures for \(\tau\) and
    \(\widetilde{\tau}\).
	The sorting steps involving the letters greater than \(e\)
	are precisely the steps involving \(h_1,\ldots,h_y\).
	These steps give the same total contribution in the computations of
	\(\textsf{sor}(\tau)\) and \(\textsf{sor}(\widetilde{\tau})\).
	Indeed, in the construction of \(\widetilde{\tau}\),
	the inserted letter \(e\) lies to the left of \(h_1,\ldots,h_y\).
	Thus, for each \(h_j\), \(1\leqslant j\leqslant y\),
	both its current position and its proper position are shifted one place to the right.
	Moreover, the cyclic replacement affects only the letters \(g_1,\ldots,g_x\).
	Hence the sorting steps involving \(h_1,\ldots,h_y\)
	have the same distance contributions in the two computations.
	
    Recall that \(\widetilde{\tau}\) is obtained from \(\tau\) by inserting \(e\) immediately before the final block and then applying the cyclic replacement to the letters \(f_1,f_2,\ldots,f_x\). 
    Since
    \[
    \tau_{(e)}=P\,f_1f_2\cdots f_x,
    \]
    after the letters \(h_1,\ldots,h_y\) have been sorted and deleted, 
    we have
    \[
    \widetilde{\tau}_{(e)}
    =
    P\,e\,f_2f_3\cdots f_xf_1.
    \]
  
	The next sorting step for \(\widetilde{\tau}\) moves \(e\) to its proper
	position.
	Since there are exactly \(x\) letters to the right of \(e\) in
	\(\widetilde{\tau}_{(e)}\),
	this step contributes \(x\).
	After this step,
	\[
	P\,e\,f_2f_3\cdots f_xf_1
	\]
	becomes
	\[
	P\,f_1f_2\cdots f_x\,e.
	\]
	Apart from the fixed letter \(e\),
	the remaining sequence is exactly
	\[
	P\,f_1f_2\cdots f_x=\tau_{(e)}.
	\]
	Thus all subsequent sorting steps are the same as the corresponding remaining steps for \(\tau\).
	It follows that
	\[
	\textsf{sor}(\widetilde{\tau})
	=
	\textsf{sor}(\tau)+x.
	\] 
	This proves the lemma.
\end{proof}

We are now ready to prove \eqref{eq-(exc,sden)-Exc-sden}, which we state as the following lemma.
\begin{lemma}\label{lemma:Exc-sden}
	For $\sigma\in\mathfrak{S}_{n-1}$ and $0\leqslant c\leqslant n-1$,
	we have
	\begin{align*}
		\emph{\textsf{Exc}}(\phi_{n}^{\emph{\text{sden}}}(\sigma,c))&=
		\left\{
		\begin{aligned}
			&\emph{\textsf{Exc}}(\sigma),
			\quad\quad\quad\quad~\text{if~}0\leqslant c\leqslant \emph{\textsf{exc}}(\sigma),\\
			&\emph{\textsf{Exc}}(\sigma)\cup\{k_{d}\},
			\quad\text{~otherwise,}
		\end{aligned}
		\right. \\
		\emph{\textsf{sden}}(\phi_{n}^{\emph{\text{sden}}}(\sigma,c))&=\emph{\textsf{sden}}(\sigma)+c,
	\end{align*}
	where, in the second case of the first formula, $d=c-\emph{\textsf{exc}}(\sigma)$ and
	\(k_d\) is the \(d\)-th smallest non-excedance position of \(\sigma\).
\end{lemma}

\begin{proof}
	We keep the notation used in the construction of
	$\phi_{n}^{\text{sden}}$.
	Let
	\[
	w=\phi_{n}^{\text{sden}}(\sigma,c)
	\]
	and put $s=\textsf{exc}(\sigma)$.
	Then Cases 1--3 correspond respectively to
	$c=0$,
	$1\leqslant c\leqslant s$,
	and $s+1\leqslant c\leqslant n-1$.
	
	\medskip
	\noindent\textbf{Case 1:} $c=0$.
	In this case,
	\[
	w=\sigma_1\sigma_2\ldots\sigma_{n-1}n.
	\]
	Since the letter $n$ is inserted at the last position,
	it is not an excedance.
    Moreover, 
    \(\EXCL(w)=\EXCL(\sigma)\),
    and \(\NEXCL(w)\) is obtained from \(\NEXCL(\sigma)\) by appending \(n\).
    Since \(n\) is the largest letter and is already in its proper position, the sorting index of the non-excedance-letter subsequence is unchanged.
	Hence
	\[
	\textsf{Exc}(w)=\textsf{Exc}(\sigma)
	\]
	and
	\[
	\textsf{sden}(w)=\textsf{sden}(\sigma)=\textsf{sden}(\sigma)+c.
	\]
	Thus the result holds in this case.
	
	\medskip
	\noindent\textbf{Case 2:} $1\leqslant c\leqslant s$.
	Recall that $e$ is the excedance letter extracted in Step 1,
    and that \(x\) is the number of non-excedance letters of \(\sigma\) that lie weakly to the right of position \(e\) in \(\sigma\) and are less than \(e\).
	Let $\sigma'$ be the sequence obtained after Step 1.
	We first claim that
	\begin{equation}\label{eq-Excl-increase-sden}
		\textsf{inv}(\EXCL(\sigma'))
		=
		\textsf{inv}(\EXCL(\sigma))+c-x.
	\end{equation}
	   
    In the excedance-letter subsequence,
    Step 1 first removes \(e\), and then inserts \(n\) in the position formerly occupied by the excedance letter immediately following the space labeled with \(c\).
    Thus the newly inserted letter \(n\) has exactly \(c-1\) excedance letters to its right,
    and therefore creates exactly \(c-1\) inversions in the excedance-letter subsequence.    
    
    On the other hand,
    before the removal of \(e\), 
    the inversions involving \(e\) in the excedance-letter subsequence are exactly those formed by \(e\) together with the excedance letters lying to its left and greater than \(e\).
    Indeed, by the stopping rule,
    there is no excedance letter occurring strictly between the position occupied by the letter \(e\) and the position \(e\).
    Hence every excedance letter to the right of the letter \(e\) lies in a position at least \(e\),
    and therefore has value greater than \(e\).
    Thus no such letter forms an inversion with \(e\).

    We now show that the number of excedance letters to the left of the letter \(e\) that are greater than \(e\) is \(x-1\).
    By Dumont's identity \cite{Dumont-1974}, we have
    \begin{align}\label{eq-Dumont}
    |\{\sigma_i:\sigma_i<e\leqslant i\}|
    =
    |\{\sigma_i:\sigma_i\geqslant e>i\}|.
    \end{align}
    Indeed, after adding the common term
    \(|\{\sigma_i:\sigma_i\geqslant e,\ e\leqslant i\}|\), the two sides become
    \(|\{\sigma_i:e\leqslant i\}|\) and
    \(|\{\sigma_i:\sigma_i\geqslant e\}|\), respectively.
    These two sets have the same cardinality, since \(\sigma\) is a permutation.
    The left-hand side of (\ref{eq-Dumont}) is precisely \(x\).
    Indeed, if \(\sigma_i<e\leqslant i\), then \(\sigma_i<i\), so \(\sigma_i\) is a non-excedance letter; moreover, it lies weakly to the right of position \(e\) and is less than \(e\).
    This is exactly the set counted by \(x\).
    By (\ref{eq-Dumont}), the right-hand side also has cardinality \(x\).
    For the right-hand side,
    if \(\sigma_i\geqslant e>i\),
    then \(\sigma_i\) is an excedance letter.
    Thus the right-hand side counts the excedance letters lying to the left of position \(e\) whose values are at least \(e\).
    Among these letters, the letter \(e\) is the rightmost one,
    because, by the stopping rule, no excedance letter lies strictly between the position occupied by the letter \(e\) and the position \(e\).
    Hence the number of excedance letters to the left of the letter \(e\) that are greater than \(e\) is \(x-1\), as desired.
    (Note that \(x\geqslant1\), since the letter \(e\) itself is counted on the right-hand side.)
    
    Therefore the removal of \(e\) deletes exactly \(x-1\) inversions from the excedance-letter subsequence, 
    and the insertion of \(n\) creates exactly \(c-1\) inversions.
    This proves \eqref{eq-Excl-increase-sden}.
	
	Next, Step 1 does not change the excedance set.
	In Step 2,
	the letter \(e\) is inserted as a non-excedance letter,
	and the shifted non-excedance letters remain non-excedance letters.
	Step 3 only cyclically replaces the letters \(f_1,\ldots,f_x\).
	These letters are all less than \(e\), and after the replacement they still occupy positions weakly to the right of position \(e\). 
	Hence they remain non-excedance letters, 
	so Step 3 creates no new excedance.
	Therefore
	\[
	\textsf{Exc}(w)=\textsf{Exc}(\sigma).
	\]
	
	By Lemma \ref{lemma-sor-increase-case2}, we have
	\[
	\textsf{sor}(\NEXCL(w))=\textsf{sor}(\tau)+x.
	\]
	Moreover,
	Steps 2 and 3 do not change the excedance-letter subsequence obtained after Step 1.
	Combining this with \eqref{eq-Excl-increase-sden},
	we get
	\begin{align*}
		\textsf{sden}(w)
		&=
		\sum_{i\in\textsf{Exc}(w)} i
		+\textsf{inv}(\EXCL(w))
		+\textsf{sor}(\NEXCL(w))\\
		&=
		\sum_{i\in\textsf{Exc}(\sigma)} i
		+\left(\textsf{inv}(\EXCL(\sigma))+c-x\right)
		+\left(\textsf{sor}(\tau)+x\right)\\
		&=
		\textsf{sden}(\sigma)+c.
	\end{align*}
	Thus the result holds in Case 2.
	
	\medskip
	\noindent\textbf{Case 3:} $s+1\leqslant c\leqslant n-1$.
	Let \(d=c-s\), 
	and let \(k_d\) be the \(d\)-th smallest non-excedance position of \(\sigma\).
	Suppose that there are \(u\) excedance positions of \(\sigma\) to the left of position \(k_d\) and \(v\) excedance positions of \(\sigma\) to the right of position \(k_d\).
	Then
	\[u+v=s.\]
	Since \(k_d\) is the \(d\)-th non-excedance position,
	among the first \(k_d\) positions there are \(u\) excedance positions and \(d\) non-excedance positions.
	Hence
	\[k_d=u+d.\]
	
    In Case 3, by construction,
    the letter \(n\) is inserted at position \(k_d\).
    Since \(k_d\leq n-1\), 
    this creates a new excedance at \(k_d\).	
    All original excedances remain excedances,
	and all shifted non-excedance letters remain non-excedance letters.
	Therefore
	\[\textsf{Exc}(w)=\textsf{Exc}(\sigma)\cup\{k_d\}.\]

    The new excedance letter \(n\) has exactly \(v\) excedance letters to its right.
    Since \(n\) is larger than all other letters,
    it creates exactly \(v\) new inversions in the excedance-letter subsequence.
    Thus
    \[\textsf{inv}(\EXCL(w))=\textsf{inv}(\EXCL(\sigma))+v.\]
    On the other hand,
    the non-excedance-letter subsequence is unchanged.
    Hence
    \[
    \textsf{sor}(\NEXCL(w))
    =
    \textsf{sor}(\NEXCL(\sigma)).
    \] 
    Therefore
    \begin{align*}
	\textsf{sden}(w)
	&=
	\textsf{sden}(\sigma)+k_d+v\\
	&=
	\textsf{sden}(\sigma)+u+d+v\\
	&=
	\textsf{sden}(\sigma)+s+d\\
	&=
	\textsf{sden}(\sigma)+c.
    \end{align*}
Thus the result also holds in Case 3.
\end{proof}

\begin{lemma}\label{lemma:bijection-sden}
	The map
	\[
	\phi_{n}^{\emph{\text{sden}}}:
	\mathfrak{S}_{n-1}\times\{0,1,\ldots,n-1\}\rightarrow\mathfrak{S}_{n}
	\]
	is a bijection.
\end{lemma}

\begin{proof}
	For positive integers \(i\leqslant j\), write
	\[
	[i,j):=\{i,i+1,\ldots,j-1\},
	\]
	with the convention that \([i,i)=\emptyset\).
	Let \(w=w_1w_2\ldots w_n\in\mathfrak{S}_n\).
	A non-excedance letter \(w_i\) of \(w\) is called a
	\emph{critical non-excedance letter} if
	\[
	[w_i,i)\subseteq \textsf{Exc}(w).
	\]
	If \(w_i=i\),
	then \(w_i\) is critical,
	since \([w_i,i)=\emptyset\).
	Moreover,
	every permutation has at least one critical non-excedance letter:
	indeed,
	the leftmost non-excedance letter is critical.
	
	We keep the notation used in the construction of \(\phi_n^{\text{sden}}\).
	Let
	\[
	w=w_1w_2\ldots w_n=\phi_n^{\text{sden}}(\sigma,c).
	\]
	Let \(z\) be the position such that \(w_z=n\).
	Let \(a=w_t\) be the rightmost critical non-excedance letter of \(w\).
	Assume that \[\ell_1<\ell_2<\cdots<\ell_m\] are all the non-excedance positions of \(w\).
	
	\medskip
	\noindent\textbf{Distinguishing the three cases.}
	We first show that the three cases in the construction can be distinguished
	from \(w\).
	
	\smallskip
	\noindent\textbf{(D1)}
	In Case 1,
	we have \(z=n\).
	
	\smallskip
	\noindent\textbf{(D2)}
	In Case 2, we have \(z<a\).
	
	Recall that \(e\) is the excedance letter extracted in Step 1.
	The letter \(n\) is inserted into the position occupied by the first shifted excedance letter, 
	which lies strictly to the left of position \(e\).
	Thus \(z<e\).
	
	We next show that \(a=e\).
	Suppose that the letter \(e\) occupies position \(\ell_p\) in \(w\).
	Then in \(w\), \(\ell_p\) is the first non-excedance position weakly to the right of position \(e\).
	Hence
	\[
	[e,\ell_p)\subseteq \textsf{Exc}(w),
	\]
	and so \(e=w_{\ell_p}\) is a critical non-excedance letter of \(w\).
	
	Now consider any non-excedance position \(\ell_i\) of \(w\) strictly to the right of \(\ell_p\).
	Then \(i>p\), and \(\ell_{i-1}\) is the preceding non-excedance position.
	If \(w_{\ell_i}\) is not affected by the cyclic replacement in Step 3, then, just before Step 2, 
	this letter occupied the preceding non-excedance position \(\ell_{i-1}\).
	Thus
	\[
	w_{\ell_i}\leqslant \ell_{i-1}<\ell_i.
	\]
	If \(w_{\ell_i}\) is affected by the cyclic replacement in Step 3,
	then \(w_{\ell_i}\) is one of the letters \(f_1,f_2,\ldots,f_x\),
	all of which are less than \(e\).
	Hence
	\[
	w_{\ell_i}<e\leqslant \ell_p\leqslant \ell_{i-1}<\ell_i.
	\]
	In either case,
	we have
	\[
	w_{\ell_i}\leqslant \ell_{i-1}<\ell_i.
	\]
	Since \(\ell_{i-1}\) is a non-excedance position of \(w\) and
	\[
	\ell_{i-1}\in [w_{\ell_i},\ell_i),
	\]
	the letter \(w_{\ell_i}\) is not critical.
	Therefore the rightmost critical non-excedance letter of \(w\) is \(e\).
	Thus \(a=e\),
	and consequently \(z<e=a\).
	
	\smallskip
	\noindent\textbf{(D3)}
	In Case 3,
	we have \(a\leqslant z<n\).
	
	In this case,
	the letter \(n\) is inserted at the position \(k_d\).
	Note that \(k_d\) is a non-excedance position of \(\sigma\), 
	but it becomes an excedance position of \(w\).
	It is clear that \(z=k_d<n\).
	In what follows, we prove that \(a=w_t\leqslant k_d\).
	
	If \(t<k_d\),
	then, since \(a=w_t\) is a non-excedance letter of \(w\),
	we have
	\[
	a=w_t\leqslant t<k_d.
	\]
	Hence \(a<k_d\).
	
	It remains to consider the case \(t>k_d\).
	Let \(\ell_p\) be the first non-excedance position of \(w\) lying to the right of position \(k_d\).
	The letter \(w_{\ell_p}\) comes from the non-excedance letter originally at position \(k_d\) in \(\sigma\),
	and is therefore at most \(k_d\).
	Thus, if \(a=w_{\ell_p}\),
	then \(a\leqslant k_d\).
	
    For any non-excedance position \(\ell_i\) of \(w\) strictly to the right of \(\ell_p\),
    the preceding non-excedance position is \(\ell_{i-1}\).
    Since \(w_{\ell_i}\) comes from the non-excedance letter occupying
    the position \(\ell_{i-1}\) before the shift in Case 3,	
	we have
	\[
	w_{\ell_i}\leqslant \ell_{i-1}<\ell_i.
	\]
	Thus
	\[
	\ell_{i-1}\in [w_{\ell_i},\ell_i).
	\]
    Since \(\ell_{i-1}\) is a non-excedance position of \(w\),
    the letter \(w_{\ell_i}\) is not critical.
	Therefore the rightmost critical non-excedance letter \(a\) cannot be any of these later non-excedance letters.
	
	Consequently, in all cases we have \(a\leqslant k_d\).
	Since \(z=k_d\), it follows that \(a\leqslant z\).

	\medskip
	\noindent\textbf{Recovering the preimage.}
	We now recover \((\sigma,c)\) uniquely from \(w\).
    Recall that \(w_z=n\),
    and that \(a=w_t\) is the rightmost critical non-excedance letter of \(w\).

	\smallskip
	\noindent\textbf{(R1)}
	If \(z=n\),
	then \(w\) comes from Case 1.
	In this case,
	\(\sigma\) is obtained from \(w\) by deleting the letter \(n\).
	
	\smallskip
	\noindent\textbf{(R2)}
	If \(z<a\),
	then \(w\) comes from Case 2.
	Set \(e:=a\).
	Let \(\widetilde{\tau}\) be the subsequence of \(w\) consisting of the
	non-excedance letters.
	Form the sequence \(\widetilde{\tau}_{(e)}\).
	Write
	\[
	\widetilde{\tau}_{(e)}=P\,e\,q_1q_2\ldots q_r,
	\]
	where \(q_1q_2\ldots q_r\) is the part of
	\(\widetilde{\tau}_{(e)}\) to the right of the letter \(e\).
	
	In Case 2,
	before the cyclic replacement,
	the corresponding final string is
	\[
	f_1f_2\ldots f_x,
	\]
	and after the cyclic replacement it becomes
	\[
	f_2f_3\ldots f_xf_1.
	\]
	Hence \(r=x\),
	and
	\[
	q_1q_2\ldots q_r=f_2f_3\ldots f_xf_1.
	\]
	Apply the inverse cyclic replacement
	\[
	q_1\mapsto q_r,\quad
	q_2\mapsto q_1,\quad
	\ldots,\quad
	q_r\mapsto q_{r-1}
	\]
	simultaneously in \(w\), and let \(u\) be the resulting sequence.
	Then \(u\) is the sequence obtained after Step 2 in the construction.
		
    Next, we reverse Step 2.
    Remove the letter \(e\) from \(u\), leaving its position empty.
    Regard this empty position as a non-excedance position.
    Then shift the non-excedance letters lying to the right of this empty position one non-excedance position to the left.
	Finally,
	delete the empty position created at the far right.
	Let \(v\) be the resulting sequence.
	Then \(v\) is the sequence obtained after Step 1 in the construction.
	
	It remains to reverse Step 1.
	Let
	\[
	p_1<p_2<\cdots<p_k
	\]
	be all positions \(p\in [z,e)\) such that \(v_p>p\).
	Then \(p_1=z\) and \(v_{p_1}=n\).
	Replace the letter at position \(p_i\) by \(v_{p_{i+1}}\) for \(1\leqslant i\leqslant k-1\),
	and replace the letter at position \(p_k\) by \(e\).
	The resulting permutation is \(\sigma\).
	
	\smallskip
	\noindent\textbf{(R3)}
	If \(a\leqslant z<n\),
	then \(w\) comes from Case 3.
	In this case,
	we recover \(\sigma\) as follows.
	Remove the letter \(n\), leaving its position empty. 
	Regard this empty position as a non-excedance position. 
	Then shift the non-excedance letters lying to the right of this empty position one non-excedance position to the left.	
	Finally, delete the empty position created at the far right.
	The resulting permutation is \(\sigma\).
	
	Thus \(\sigma\) is uniquely determined by \(w\).
	Finally,
	set
	\[
	c=\textsf{sden}(w)-\textsf{sden}(\sigma).
	\]
	By Lemma \ref{lemma:Exc-sden},
    this recovers the original value of \(c\).
	Therefore the pair \((\sigma,c)\) is uniquely recovered from \(w\).
	Hence \(\phi_n^{\text{sden}}\) is injective.
	
	Since
	\[
	|\mathfrak{S}_{n-1}\times\{0,1,\ldots,n-1\}|
	=
	|\mathfrak{S}_n|
	=
	n!,
	\]
	the map \(\phi_n^{\text{sden}}\) is a bijection.
\end{proof}

By Lemmas \ref{lemma:Exc-sden} and \ref{lemma:bijection-sden},
we obtain Theorem \ref{Thm-bijection-exc-sden}.

\section*{Acknowledgment}
\addcontentsline{toc}{section}{Acknowledgment} 
This work was supported by the National Natural Science Foundation of China (12101134).

\phantomsection
\begin{spacing}{0.9}
	
\end{spacing}
\end{document}